\documentclass[11pt]{amsart}
\usepackage{amsthm,amssymb}
\usepackage{mathtools}
\usepackage[colorlinks]{hyperref}
\usepackage[margin=1in]{geometry}
\usepackage{tikz}
\usepackage{multirow}
\usepackage{float}
\usepackage{booktabs}

\usepackage{tikz-cd}
\usepackage{tikz}
\usetikzlibrary{snakes, 
	3d, matrix,decorations.pathreplacing,calc,decorations.pathmorphing, patterns,automata}
\usetikzlibrary {positioning}
\usetikzlibrary{matrix, fit}%color for matrix in tikz
\usetikzlibrary{backgrounds}

\newcommand{\C}{\mathbb{C}}

\newcommand{\Q}{\mathbb{Q}}

\newcommand{\Z}{\mathbb{Z}}

\newcommand{\B}{\mathcal{B}}

\newcommand{\oo}{\bf {Op.1}}
\newcommand{\ooo}{\bf {Op.2}}

\def\mathscr{\mathscr}

\def\mcal{\mathcal}

\def\block(#1,#2)#3{\multicolumn{#2}{c}{\multirow{#1}{*}{$ #3 $}}}

\newtheorem{theorem}{Theorem}[section]

\newtheorem{lemma}[theorem]{Lemma}
\newtheorem{proposition}[theorem]{Proposition}

\newtheorem{corollary}[theorem]{Corollary}

\newtheorem{conjecture}[theorem]{Conjecture}
\newtheorem{question}[theorem]{Question}

\theoremstyle{definition}
\newtheorem{example}[theorem]{Example}
\newtheorem{definition}[theorem]{Definition}
\newtheorem{remark}[theorem]{Remark}

\numberwithin{equation}{section}

\DeclareMathOperator{\Cone}{Cone}

\begin{document}
	\title{Unique toric structure on a Fano Bott manifold}
	
\begin{abstract}
	We prove that if there exists a $c_1$-preserving graded ring isomorphism between integral cohomology rings of two Fano Bott manifolds, then they are 
	isomorphic as toric varieties. As a consequence, we give an affirmative answer to McDuff's question on the uniqueness of a toric structure on a Fano Bott manifold.
\end{abstract}

\author{Yunhyung Cho}
\address{Department of Mathematics Education, Sungkyunkwan University, Seoul, Republic of Korea}
\email{yunhyung@skku.edu}

\author{Eunjeong Lee}
\address{Center for Geometry and Physics, Institute for Basic Science (IBS), Pohang, Republic of Korea}
\email{eunjeong.lee@ibs.re.kr}

\author{Mikiya Masuda}
\address{Osaka City University Advanced Mathematics Institute (OCAMI) \& Department
of Mathematics, Graduate School of Science, Osaka City University, Sumiyoshiku,
Sugimoto, 558-8585, Osaka, Japan}
\email{masuda@osaka-cu.ac.jp}

\author{Seonjeong Park}
\address{Department of Mathematical Sciences, KAIST, Daejeon, Republic of Korea}
\email{seonjeong1124@gmail.com}

\maketitle

%-----------------------------------------------------------------------------------------------------------------------------------------------------------------------------------------------------------------------------------------------------------------------------------------------------------
\section{Introduction}
	
	To each symplectic manifold $(M,\omega)$,  one can associate the Hamiltonian diffeomorphism group $\mathrm{Ham}(M,\omega)$. 
	It is a normal subgroup of the symplectomorphism group $\mathrm{Symp}(M,\omega)$ and
	governs all possible Hamiltonian Lie group actions on $(M,\omega)$. The group $\mathrm{Ham}(M,\omega)$ is infinite dimensional and non-compact 
	in general, and it might possess more than one maximal tori with distinct conjugacy classes. It turned out that the number of conjugacy classes of maximal tori in 
	$\mathrm{Ham}(M,\omega)$,
	allowing conjugations by elements of $\mathrm{Symp}(M,\omega)$, is finite; it was proved by Karshon--Kessler--Pinsonnault \cite{KKP} and Pinsonnault \cite{Pin} in dimension four, 
	and by McDuff \cite[Proposition 3.1]{McD} in any dimension.

	Recall that a symplectic form $\omega$ is called {\em monotone} if $c_1(M) := c_1(TM,J) = \lambda \cdot [\omega]$ for some $\lambda > 0$ and an 
	$\omega$-compatible almost complex structure $J$ on $M$. 
	Throughout this paper, we always assume that $\lambda = 1$ 
	unless stated otherwise.	
	In this paper, we concern the following question posed by McDuff~\cite{McD}.
	
	\begin{question}[McDuff]\label{Q_1}\cite[Question~1.11]{McD}
		Let $(M,\omega)$ be a $2n$-dimensional closed monotone symplectic manifold. If $T_1$ and $T_2$ are n-tori in $\mathrm{Ham}(M,\omega)$, are $(M,\omega,T_1)$ 
		and $(M,\omega, T_2)$ equivariantly symplectomorphic? Equivalently, is the number of conjugacy classes of $n$-tori in $\mathrm{Ham}(M,\omega)$ precisely one?
	\end{question}
	
	Due to the Delzant's theorem \cite{Del},
	any closed symplectic toric manifold is equivariantly symplectomorphic to a smooth projective toric variety equipped with a torus invariant K\"{a}hler form.
	When a symplectic form is monotone, it is equivariantly symplectomorphic to a smooth Fano toric variety by the Kleiman's ampleness criterion \cite[Theorem 1 in Section III-1]{Kl}
	and its moment polytope becomes a 
	{\em reflexive}\footnote{A polytope is {\em reflexive} if 
	it is integral and has a unique interior lattice point such that the affine distance from the point to each facet is equal to one.} polytope.

	In the algebro-geometric aspects, Question \ref{Q_1} asks whether a smooth Fano toric variety may have more than one toric structure. McDuff \cite{McD} gave an affirmative answer 
	to Question \ref{Q_1} when $M = \C P^k \times \C P^m$, and Fanoe \cite{Fa} generalized McDuff's result to the case of $\C P^k$-bundle over $\C P^m$. 
	To the best of the authors' knowledge, Question \ref{Q_1} is still open.

	This paper is addressed to Question \ref{Q_1} in case that $M$ is a Bott manifold. 
	A {\em Bott tower}, first introduced by Grossberg and Karshon~\cite{GK_Bott_towers}, is an iterated $\C P^1$-bundle starting with a point
	\[
		\B_n \stackrel{\pi_n} \longrightarrow \B_{n-1} \stackrel{\pi_{n-1}} \longrightarrow \cdots \longrightarrow \B_1 = \C P^1 \stackrel{\pi_1} \longrightarrow \B_0 = \{ \mathrm{a ~point} \}
	\]
	where each $\B_i$ is obtained by projectivizing the direct sum of the trivial line bundle $\underline{\C}$ and a complex line bundle $\xi_{i}$ over $\B_{i-1}$, i.e.,
	$\B_i = P(\underline{\C} \oplus \xi_i)$. The total space $\B_n$ is called a {\em Bott manifold}.
	
	We may equip a natural complex structure on a Bott manifold by taking each $\xi_i$ as a holomorphic line bundle
	so that $\B_n$ becomes a complex manifold with a natural $(\C^*)^n$-action 
	constructed in an iterative way using a toric structure of a base space and a $\C^*$-action on a fiber at each stage, see~\cite[Section~I-7.6${}^\prime$]{Oda2}.
	Indeed, $\B_n$ is a smooth projective toric variety.
	
	Any Bott manifold $\B_n$ can be characterized by an $n \times n$ lower triangular 
	integer matrix called a {\em Bott matrix}. Roughly speaking, a Bott manifold $\B_n$ is a smooth projective toric variety,
	where the corresponding fan is combinatorially equivalent to the normal fan of the $n$-cube. After fixing one reference maximal cone of the fan and making it into the ``first quadrant'' 
	via some basis change by multiplying a suitable element of $\mathrm{GL}(n,\Z)$, we obtain residual $n$ (column) vectors which form a Bott matrix. Note that a Bott matrix presentation 
	of a Bott manifold is not unique and it depends on the choice of a reference cone and an element of $\mathrm{GL}(n,\Z)$. In Section \ref{secOperationsOnBottMatrices}, we illustrate some 
	operations (temporarily denoted by {\oo} and {\ooo} in this paper) on the set of Bott matrices and describe cohomology ring isomorphisms induced by the operations that will 
	be crucially used in the proof of our main theorem.

	Not all Bott manifolds are Fano.  For instance, there are only two Fano Bott manifolds in dimension four: $\C P^1 \times \C P^1$ and a Hirzebruch surface 
	$P(\underline{\C} \oplus \mcal{O}(1))$. Recently, Suyama \cite{Su} classified all Fano Bott manifolds in terms of Bott matrices, see Section \ref{secBottManifolds} for the detail.
	Now we state our main theorem which says that every Fano Bott manifold is characterized (as a toric variety) by its integral cohomology ring and the first Chern class. 
	
	\begin{theorem}[Theorem~\ref{thm_main_body}]\label{thm_main}
		Let $X$ and $Y$ be Fano Bott manifolds. If there exists a $c_1$-preserving graded ring isomorphism 
		\[
			\varphi \colon H^*(X;\Z) \rightarrow H^*(Y;\Z),
		\]
		then $X$ and $Y$ are isomorphic\footnote{We say that two toric varieties $X$ and $Y$ are {\em isomorphic as toric varieties} if there exists a toric isomorphism $\phi$
		from $X$ to $Y$, i.e., $\phi(t \cdot x) = \phi(t) \cdot \phi(x)$ for every element $t$ in the torus $T_X \subset X$ and $x \in X$. In fact, it was proved in~\cite{Be} that there exists a 
		toric isomorphism if $X$ and $Y$ are isomorphic as abstract varieties.}
		as toric varieties, i.e., the fans associated to $X$ and $Y$ are unimodularly equivalent.
	\end{theorem}

	We call a monotone symplectic manifold $(M,\omega)$ a {\em monotone Bott manifold} if $M$ is diffeomorphic to a Bott manifold.
	Using Theorem \ref{thm_main}, we obtain a positive answer to Question~\ref{Q_1}.
	
	\begin{corollary}\label{cor_main}
		Any monotone Bott manifold has a unique toric structure.
	\end{corollary}
	
	\begin{proof}		
		Suppose that $(M,\omega)$ is a $2n$-dimensional monotone Bott manifold and $T_1$, $T_2$ are $n$-tori in $\mathrm{Ham}(M,\omega)$.
		From the Delzant's theorem \cite{Del}, each $T_i$-action makes $M$ into a toric Fano variety (which we denote by $X_i$) with $T_i$-invariant complex structure $J_i$ on $M$.
		Note that each $J_i$ can be chosen to be $\omega$-compatible so that $c_1(X_1) = c_1(X_2) = [\omega] \in H^2(M; \Z)$.
		
		On the other hand, it follows from~\cite[Corollary 3.5 and Theorem~5.5]{MP} that any smooth projective toric variety whose integral cohomology ring is isomorphic to that of a Bott manifold is in fact isomorphic 
		to a Bott manifold as a toric variety. (See Remark \ref{rem_MP} for the detail.) Thus we may assume that $X_1$ and $X_2$ are Fano Bott manifolds.
		Then the identity map $H^*(M;\Z) \rightarrow H^*(M;\Z)$ induces a graded ring isomorphism 
		\[
			H^*(X_1; \Z) \rightarrow H^*(X_2; \Z), \quad \quad c_1(X_1) = [\omega] \mapsto [\omega] = c_1(X_2).
		\]
		Thus the result follows from Theorem~\ref{thm_main}.
	\end{proof}
	
	It is worth mentioning a relation between Theorem \ref{thm_main} and a problem posed by the third author and Suh \cite[Problems~1 and~4]{MS}
	which asks whether two smooth complete toric varieties having isomorphic cohomology rings (as graded rings) are diffeomorphic or not. This problem is now called
	the {\em cohomological rigidity} for toric varieties.
	There are many partial affirmative answers to the problem. For instance, two smooth complete toric varieties with Picard number~$2$ are diffeomorphic if and only if their integral 
	cohomology rings are isomorphic as graded rings, see~\cite{CMS10_trans}. We also refer the reader to~\cite{CMS11,BEMPP} and references therein for recent accounts of this problem.
	
	Inspired by Theorem \ref{thm_main}, we pose the following conjecture.
	
	\begin{conjecture}\label{conj_Fano_rigidity}
		Suppose that $X$ and $Y$ are smooth  toric Fano varieties. If there exists a $c_1$-preserving graded ring isomorphism between their integral cohomology rings, 
		then $X$ and $Y$ are isomorphic as toric varieties.  
	\end{conjecture}
	
	Conjecture \ref{conj_Fano_rigidity} was verified for some other classes of smooth toric Fano varieties. 
	Indeed, the authors confirmed Conjecture \ref{conj_Fano_rigidity} for smooth toric Fano varieties 
	with Picard number 2, whose proof will be provided in an upcoming manuscript \cite{CLMP}.
	Also the third author together with Higashitani and Kurimoto \cite{HKM} proved Conjecture \ref{conj_Fano_rigidity} 
	for smooth toric Fano varieties with small dimension ($\dim X_\C \leq 4$) or with large Picard number.)
	
	Note that if Conjecture \ref{conj_Fano_rigidity} is true, then the answer to Question \ref{Q_1} is positive. More precisely, if $(M,\omega,T_1)$ and $(M,\omega,T_2)$ are two 
	toric structures over the same monotone symplectic manifold $(M,\omega)$, then the identity map on $H^*(M;\Z)$ satisfies the hypothesis in Conjecture \ref{conj_Fano_rigidity}. 
	Therefore, Conjecture \ref{conj_Fano_rigidity} can be thought as a stronger version of Question ~\ref{Q_1}.

	As a final remark, we would like to mention a recent work of Pabiniak and Tolman. In  \cite{PT}, they considered the following question which they called 
	{\em symplectic cohomological rigidity}.
	
	\begin{question}\label{conj_symplectic_rigidity}\cite[p.3]{PT}
		Let $(M_1,\omega_1)$ and $(M_2,\omega_2)$ be symplectic toric manifolds. If there exists a graded ring isomorphism between their integral cohomology rings 
		sending $[\omega_1]$ to $[\omega_2]$, are $(M_1,\omega_1)$ and $(M_2,\omega_2)$ symplectomorphic?
	\end{question}
	They also gave a positive answer to Question \ref{conj_symplectic_rigidity} under the assumptions that 
	$\omega_1$ and $\omega_2$ are rational symplectic forms and that $H^*(M_1;\Q) \cong H^*(M_2;\Q) \cong H^*(\C P^1 \times \cdots \times \C P^1;\Q)$.
	
	This paper is organized as follows. In Section~\ref{secBottManifolds}, we explain the notion of Bott manifolds and also discuss their cohomological properties. 
	In Section~\ref{secOperationsOnBottMatrices}, we introduce two operations {\oo} and {\ooo} on Bott matrices 
	and prove that any two Bott matrices which represent isomorphic Bott manifolds  
	are obtained by applying those operations repeatedly. In Section~\ref{secMainTheorem}, we give the proof of Theorem~\ref{thm_main}. 
	
	\subsection*{Acknowledgements} 
		Y. Cho was supported by the National Research Foundation of Korea(NRF) grant funded by the Korea government(MSIP; Ministry of Science, ICT \& Future Planning) 				(NRF-2020R1C1C1A01010972). E. Lee was supported by IBS-R003-D1.
	M. Masuda was supported in part by JSPS Grant-in-Aid for Scientific Research 16K05152.
	S. Park was supported by the Basic Science Research Program through the National Research Foundation of Korea (NRF) funded by the Government of Korea
		(NRF-2018R1A6A3A11047606).  
		This work was partly supported by Osaka City University Advanced Mathematical Institute (MEXT Joint Usage/Research Center on Mathematics and Theoretical Physics).

%-----------------------------------------------------------------------------------------------------------------------------------------------------------------------------------------------------------------------------------------------------------------------------------------------------------	
\section{Bott manifolds}
\label{secBottManifolds}

We begin by recalling the definition of Bott towers and Bott manifolds.
\begin{definition}\cite[\S 2.1]{GK_Bott_towers}
	A \emph{Bott tower} $\B_\bullet$  is   an iterated $\C P^1$-bundle starting with a point:
	\begin{equation}\label{eq_Bott_tower}
		\begin{tikzcd}[row sep = 0.2cm]
			\mathcal{B}_{n} \arrow[r, "\pi_{n}"] \arrow[d, equal]&  
			\mathcal{B}_{n-1} \arrow[r, "\pi_{n-1}"] &
			\cdots \arrow[r, "\pi_2"] &
			\mathcal{B}_1 \arrow[r, "\pi_1"] \arrow[d, equal]&
			\mathcal{B}_0, \arrow[d, equal]\\
			P(\underline{\C}\oplus \xi_{n}) & & & \C P^1 & \{\text{a point}\}
		\end{tikzcd} 
	\end{equation}
	where each $\mathcal{B}_{i}$ is the complex projectivization of the Whitney sum of a holomorphic line bundle $\xi_{i}$ and the trivial bundle $\underline{\C}$ over $\mathcal{B}_{i-1}$.
	The total space $\B_n$ is called a {\em Bott manifold}.
\end{definition}

	Let $\gamma_j$ be the tautological line bundle over $\mathcal{B}_j$ and $\gamma_{i,j}$ the pullback of $\gamma_j$ by the projection 
	$\pi_i\circ \cdots \circ\pi_{j+1}\colon \mathcal{B}_{i}\to \mathcal{B}_j$ for $i>j$. We also define $\gamma_{j,j} \coloneqq\gamma_j$ for convenience. The Picard group of $\B_{i-1}$ is isomorphic to the free abelian group of rank $i-1$, and is generated by the line bundles $\gamma_{i-1,j}$ for $1 \leq j < i$ by~\cite[Exercise~II.7.9]{Har77}.	
	Therefore, for each $i=2,\ldots,n$, there exist $a_{i,j} \in \Z$ for 
	$1 \leq j < i$ such that	
	\begin{equation}\label{eq_xi_and_aij}
		\xi_{i}=\bigotimes_{1\leq j < i}\gamma_{i-1,j}^{\otimes a_{i,j}}.
	\end{equation}
	Thus the set of integers $\{ a_{i,j} \}_{1 \leq j < i \leq n}$ determines a Bott manifold.

	Each projection $\pi_{i} \colon \B_{i} \rightarrow \B_{i-1}$ admits a section induced from the zero section of $\underline{\C} \oplus \xi_{i}$. This implies that 
	\[
		\pi_i^* \colon H^*(\B_{i-1}; \Z) \rightarrow H^*(\B_i; \Z)
	\]
	is an injective ring homomorphism. By abuse of notation, we continue to write $x_j  \in H^2(\B_{i-1}; \Z)$ for the first Chern class of the dual of $\gamma_{i-1,j}$ for each $i > j$. 
	From \eqref{eq_xi_and_aij}, we obtain
	\[
		c_1(\xi_{i})= - \sum_{j=1}^{i-1}a_{i,j}x_j\in H^2(\B_{i-1}).
	\]
	
	On the other hand, a Bott manifold $\B_n$ is a smooth projective toric variety by the construction (cf.~\cite{GK_Bott_towers} and~\cite[Section~I-7.6${}^\prime$]{Oda2}).	
	If $\B_n$ is obtained from $\{ a_{i,j} \}_{1 \leq j < i \leq n}$, then it is known from~\cite[\S 3]{Civan} that its fan has $2n$ rays and their generators 
	are column vectors of the following matrix 
	\begin{equation}\label{eq_Bott_matrix}
		(E~|~A)\coloneqq\begin{bmatrix}
			1&   &  &           &  & -1            &                 &            &                         &    \\
			  & 1&  &           &  & a_{2,1}   &            -1  &            &                         &    \\
			  &   &1&           &  & a_{3,1}   & a_{3,1}     &  -1       &                         &    \\
			  &   &  & \ddots &  &\vdots      & \vdots       & \ddots &\ddots               &    \\
			  &   &  &            &1& a_{n,1} &  a_{n,2} &\cdots   & a_{n,n-1}   & -1\\
				\end{bmatrix},
	\end{equation}
	where $E$ is the $n \times n$ identity matrix. We call an integer matrix of the form $A$ in \eqref{eq_Bott_matrix} a {\em Bott matrix}.
	For a given Bott matrix $A$, we denote by
	\[
	{\bf e}_j \coloneqq \text{$j$-th standard basis vector}, \quad {\bf v}_j \coloneqq \text{$j$-th column vector of $A$}.
	\]
The fan $\Sigma_A$ of the Bott manifold $\B_n$  has $2^n$ number of maximal cones $\Sigma(n) = \{\sigma_I \mid I \subset [n] \}$, where $[n] \coloneqq \{1,\dots,n\}$ and 
	\begin{equation}\label{eq_fan}
	\sigma_I = \Cone\left(\{ {\bf e}_j \mid i \in I \} \cup \{ {\bf v}_j \mid j \in I^c \} \right).
	\end{equation}
	%For later use, we denote by $A_\B\coloneqqA$ and $\hat{A}_\B\coloneqqA_\B+E$.

\begin{example}
	Let $n = 2$. Then the Bott manifold determined by $\{a_{2,1}\}$ is a Hirzebruch surface $\mathcal{H}_{a_{2,1}} \coloneqq P(\underline{\C} \oplus \mathcal{O}(-a_{2,1}))$. 
There are four maximal cones:
	\[
	\sigma_{\emptyset} = \Cone\{ \mathbf{v}_1, \mathbf{v}_2\}, \quad \sigma_{\{1\}} = \Cone \{ \mathbf{e}_1, \mathbf{v}_2\},\quad 
	\sigma_{\{2\}} = \Cone\{\mathbf{e}_2, \mathbf{v}_1\}, \quad \sigma_{\{1,2\}} = \Cone\{ \mathbf{e}_1, \mathbf{e}_2\}. 
	\]
	We present the fan of $\mcal{H}_1$ in Figure~\ref{fig_B2}. 
	\begin{figure}[t]
		\begin{tikzpicture}[scale = 0.7]
		\draw[gray,->] (-2,0)--(2,0);
		\draw[gray, ->] (0,-2)--(0,2);
		
		\fill[pattern color = blue!50, pattern =horizontal lines, semitransparent] (0,0)--(2,0)--(2,2)--(0,2)--cycle;
		\fill[pattern color = blue!50!red, pattern = north west lines, semitransparent] (0,0)--(0,2)--(-2,2)--cycle;
		\fill[pattern color = red!50, pattern = vertical lines, semitransparent] (0,0)--(-1,1)--(-2,2)--(-2,-2)--(0,-2)--cycle;
		\fill[pattern color = green!50!black, pattern = north east lines, semitransparent] (0,0)--(2,0)--(2,-2)--(0,-2)--cycle;
		
		\draw[->, thick] (0,0)--(1,0) node[ below right ] {$\mathbf{e}_1$};
		\draw[->, thick] (0,0)--(0,1) node[ right] {$\mathbf{e}_2$};
		\draw[->, thick] (0,0)--(-1,1) node[above left] {$\mathbf{v}_1 $};
		\draw[->, thick] (0,0)--(0,-1) node[below] {$\mathbf{v}_2$};
		
		\node[fill=white, draw=blue] at (1.8,1.5) {$\sigma_{\{1,2\}}$};
		\node[fill=white, draw=blue!50!red] at (-0.8,2) {$\sigma_{\{2\}}$};
		\node[fill=white, draw=red] at (-1.5,0) {$\sigma_{\emptyset}$};
		\node[fill=white, draw = green!50!black] at (1.8,-1.5) {$\sigma_{\{1\}}$};
		\end{tikzpicture}	
		\caption{The fan of $\mcal{H}_{1}$.}\label{fig_B2}
	\end{figure}
\end{example}

Let $\mathcal{M}_{n}$ be the set of all Bott matrices of size $n\times n$, i.e., the set of all $n\times n$ lower triangular integer matrix with $-1$'s on the main diagonal as in \eqref{eq_Bott_matrix}. 
Since a Bott matrix $A$ determines the fan $\Sigma_A$ of a Bott manifold, we denote the corresponding Bott manifold by $\B(A)$.  Note that it happens that $\B(A)$ and $\B(A')$ are isomorphic as toric varieties even if $A$ and $A'$ are different.

\begin{remark}\label{rem_MP}
	For a given smooth projective toric variety $M$, if its integral cohomology ring  is isomorphic to that of a certain Bott manifold $\B$ as graded rings, then the fan of $M$ is combinatorially equivalent to that of $\B$, i.e., it is combinatorially equivalent to the normal fan of the $n$-cube  (see~\cite[Theorem~5.5]{MP}).  Moreover, such a toric variety $M$ is again a Bott manifold by~\cite[Corollary~3.5]{MP}. 
	Accordingly, any smooth projective toric variety whose integral cohomology ring is isomorphic to that of a Bott manifold is isomorphic to a Bott manifold. 
%	
%	A fan $\Sigma_A$ associated to a Bott manifold $\B(A)$ consists of $2n$ rays $\rho_1,\dots,\rho_{2n}$ and admits $n$ primitive 
%	collections\footnote{
%	Denote by $\Sigma(1)$ the set of 1-dimensional cones of $\Sigma$ and $\sigma(1)$ the set of 1-dimensional faces of $\sigma$ for $\sigma \in \Sigma$. 
%	A subset $P = \{\rho_1,\dots,\rho_k\} \subset \Sigma(1)$ of rays is called a \textit{primitive collection} if $P$ is not contained in $\sigma(1)$ for all $\sigma \in \Sigma$ but any proper subset is.} $\{ {\rho}_i, \rho_{i+n}\}$ for $i =1,\dots,n$.
%	 One may wonder whether a smooth projective toric variety whose fan admits the same primitive collections as those of $\B(A)$ 
%	 is a Bott manifold. This question was answered in~\cite[Theorem~3.4]{MP}. 
%	 More precisely, let $\Sigma$ be a fan of a smooth projective toric variety consisting of $2n$ rays $\rho_1,\dots,\rho_{2n}$ and $n$ primitive collections $\{ \rho_i, \rho_{i+n}\}$ for $i=1,\dots,n$. 
%	We may assume that the ray generators of rays $\rho_1,\dots,\rho_n$ form a $\Z$-basis of $\Z^n$ since it forms a maximal cone in the fan $\Sigma$. Let $A$ be the matrix whose column vectors are the ray generators of $\rho_{n+1},\dots,\rho_{2n}$. The matrix $A$ does not need to be a lower triangular matrix. 
%	Then, the toric variety $X_{\Sigma}$ is a Bott manifold if and only if 
%	every principal minor (even the determinant of the matrix) of $A$ is $(-1)^d$ where $d$ is the size of the minor.
\end{remark}

%%%%
\subsection{Cohomology rings}
\label{ssecCohomologyRings}

By the Borel--Hirzebruch formula~\cite{BH58}, the integral cohomology ring of a Bott manifold $\B(A)$ for $A\in \mathcal{M}_{n}$ is described by
	\begin{equation}\label{eq_coh_bott}
		\begin{split}
			H^\ast(\B(A); \Z)&\cong \Z[x_1,\dots,x_n]/\langle x_i^2 + c_1(\xi_i)x_i\mid i=1,\dots,n\rangle\\
			&\cong \Z[x_1,\dots,x_n]/\langle x_i^2-(a_{i,1}x_1+\dots+a_{i,i-1}x_{i-1})x_i \mid i=1,\dots,n\rangle \\
			& \cong \Z[x_1,\dots,x_n]/\langle x_i^2-\alpha_ix_i\mid i=1,2,\dots,n \rangle.
		\end{split}
	\end{equation}
	Here, $x_i$ is the first Chern class of the dual of $\gamma_{n,i}$, and we set
\[
	\alpha_i \coloneqq a_{i,1}x_1+\dots+a_{i,i-1}x_{i-1} \in H^2(\B(A);\Z).
\]
Note that $x_1,\dots,x_{n}$ are of degree two and they generate the cohomology ring $H^{\ast}(\B(A); \Z)$.

For a $\Q$-coefficient, we often use the following notation 
\begin{equation}\label{eq_y}
	y_i \coloneqq x_i- \frac{1}{2}\alpha_i \in H^2(\B(A);\Q).
\end{equation}
Note that $y_i$'s may not be integral classes but they generate $H^*(\B(A); \Q)$ as a ring.

Recall from~\cite[Theorem~3.12]{Oda} that the total Chern class of a Bott manifold $\B(A)$ is written by
\[
	c(\B(A))=\prod_{i=1}^n(1+x_i)(1+x_i-\alpha_i)=\prod_{i=1}^n(1+2x_i-\alpha_i).
\]
Substituting~\eqref{eq_y} in the above, we get
\begin{equation} \label{eq_total_Chern}
	c(\B(A))=\prod_{i=1}^n(1+2y_i)
\end{equation}
which implies that 
\begin{equation} \label{eq_first_Chern}
c_1(\B(A))=2\sum_{i=1}^n{y_i}\quad\text{ and} \quad c_n(\B(A))=2^n\prod_{i=1}^n y_i.
\end{equation}

\begin{remark}\label{rem_linearly_independent}
Note that the cohomology ring description in~\eqref{eq_coh_bott} can be obtained from the Danilov--Jurkiewicz theorem, see~\cite[Theorem~5.3.1]{BP_toric}.
It follows that
 the even Betti number $b_{2i}(\B(A))$ is ${n \choose i}$ for $1 \leq i \leq n$. Since $H^*(\B(A);\Z)$ is generated by degree two elements $\{x_i \mid 1 \leq i \leq n\}$, we obtain the following.
	\begin{enumerate}
		\item Since $x_1,\dots,x_n$ are linearly independent (over $\Z$), so are $y_1, \dots, y_n$ (over $\Q$).
		\item The set $\{ x_{i_1} \cdots x_{i_k} \mid 1 \leq i_1 < \cdots < i_k \leq n \}$ is a $\Z$-basis of $H^{2k}(\B(A); \Z)$.  
		\item $\prod_{i=1}^ny_i$ is the orientation class of $\B(A)$. Indeed, since the Euler characteristic 
		 of $\B(A)$ is $2^n$ and it agrees with $c_n(\B(A))$ evaluated on the fundamental class of $\B(A)$, $\prod_{i=1}^n y_i$ evaluated on the fundamental class is $1$. 
	\end{enumerate}

\end{remark}

In terms of $y_i$'s, we can obtain a simple description of (any) isomorphisms of $H^*(\B(A); \Q)$ as follows.
In the following, we denote $x_{i}$, $y_{i}$ and $\alpha_{i}$ in $H^{\ast}(\B(A))$ by $x_{i}^{A}$, $y_{i}^{A}$ and $\alpha_{i}^{A}$, respectively. %Then the following is proved in \cite[Proposition 4.1]{CMM}. 

\begin{proposition}\cite[Proposition 4.1]{CMM}\label{prop_isomorphism_y}
For $A$ and $A'$ in $\mathcal{M}_{n}$, if we have a graded ring isomorphism
\[
	\varphi\colon H^*(\B(A);\Q)\to H^*(\B(A');\Q),
\]
then there are nonzero $q_1,\dots,q_n\in \Q$ and a permutation $\sigma$ on $[n]$ such that 
\[
\varphi(y_i^A)=q_iy_{\sigma(i)}^{A'} \qquad \text{for $i=1,\dots,n$}.
\]
\end{proposition}

\begin{theorem}\label{thm_q}
Suppose that there is a $c_1$-preserving graded cohomology ring isomorphism between two Bott manifolds. Then all $q_i$'s in Proposition \ref{prop_isomorphism_y} are equal to $1$. 
Moreover, it preserves their total Chern classes and hence 
all the Chern numbers of the two Bott manifolds are the same. 
\end{theorem}

\begin{proof}
	Let $\B$ and $\B'$ be Bott manifolds determined by Bott matrices $A$ and $A'$, respectively.
	Let $\varphi$ be the $c_1$-preserving graded cohomology ring isomorphism between  $\B$ and $\B'$. 
	By \eqref{eq_first_Chern}, we have
	\[
		\varphi\left(2\sum_{i=1}^n y_i^A \right) =2\sum_{i=1}^n y_i^{A'}.
	\]
	On the other hand, it follows from Proposition~\ref{prop_isomorphism_y} that
	\[
		\varphi\left(2\sum_{i=1}^n y_i^A \right)=2\sum_{i=1}^n q_iy_{\sigma(i)}^{A'}.
	\]
	Comparing these two identities, we obtain 
	\[
		2\sum_{i=1}^n y_i^{A'}=2\sum_{i=1}^n q_iy_{\sigma(i)}^{A'}.
	\]
	Here, $y_1^{A'}, \dots,y_n^{A'}$ are linearly independent, so we conclude $q_i=1$ for any $i$. This together with~\eqref{eq_total_Chern} shows that $\varphi$ preserves their total Chern classes,
	 proving the former part of the theorem. 
	 
	 The latter part of the theorem follows from the former part and the fact that $\varphi$ preserves the orientation classes $\prod_{i=1}^n y_i^A$ and $\prod_{i=1}^ny_i^{A'}$
	 (as well as top Chern classes). 
\end{proof}

\begin{remark} We note the following.
\begin{enumerate}
	\item Not every graded cohomology ring isomorphism between Bott manifolds is $c_1$-preserving.  For instance, one can find such an isomorphism for Hirzebruch surfaces 
	$\mcal{H}_0$ and $\mcal{H}_2$.  See Example~\ref{ex-hir}.
	\item Recall that two Hirzebruch surfaces $\mcal{H}_a$ and $\mcal{H}_b$ are isomorphic if and only $|a|=|b|$. However, for any integers $a$ and $b$ with the same parity, there is a $c_1$-preserving graded cohomology ring isomorphism between Hirzebruch surfaces $\mcal{H}_a$ and $\mcal{H}_b$.  
	Therefore, the existence of such a cohomology ring isomorphism does not imply that two varieties are isomorphic.
	We would need to restrict our attention to Fano Bott manifolds to conclude a variety isomorphism. See Remark ~\ref{rem_weak_Fano}.
\end{enumerate}
\end{remark}

\begin{example}\label{ex-hir}
It follows from~\eqref{eq_coh_bott} that
\begin{equation*}
H^{\ast}(\mathcal{H}_{0};\Z)=\Z[x_{1},x_{2}]/\langle x_{1}^{2},x_{2}^{2}\rangle\qquad\text{ and }\qquad H^{\ast}(\mathcal{H}_{2};\Z)=\Z[x_{1}',x_{2}']/\langle (x_{1}')^{2},{x_{2}'}(x_{2}'-2x_{1}')\rangle.
\end{equation*} Note that $c_{1}(\mathcal{H}_{0})=2x_{1}+2x_{2}$ and $c_{1}(\mathcal{H}_{2})=2x_{2}'$.
Then the map $\varphi\colon H^{\ast}(\mathcal{H}_{0})\to H^{\ast}(\mathcal{H}_{2})$ given by $\varphi(x_{1})=x_{1}'$ and $\varphi(x_{2})=x_{1}'-x_{2}'$ is a graded ring isomorphism which does not preserve the first Chern class. On the other hand, the map $\varphi'$ defined by $\varphi'(x_1):=x_1'$ and $\varphi'(x_2):=x_2'-x_1'$ is a $c_1$-preserving isomorphism from $H^\ast(\mathcal{H}_0)$ to $H^\ast(\mathcal{H}_2)$. (Note that $\varphi'(c_1(\mathcal{H}_0))=\varphi'(2x_1+2x_2)=2x_2'=c_1(\mathcal{H}_2).$)
\end{example}

%%%%
\subsection{Fano Bott manifolds}
\label{ssecFanoBottManifolds}

In this subsection, we recall  a description of Fano Bott manifolds from~\cite{Su}.

\begin{theorem}\cite[Theorem 8]{Su}\label{thm_Fano}
	A Bott manifold $\B(A)$ is Fano if and only if each column of $A+E$ has values in $\{-1,0,1\}$ and it satisfies	
	 one of the following:
	\begin{enumerate}
		\item all entries are zero,
		\item there is at most one $1$ and every other entry below the $1$ vanishes \textup{(}if there is $1$ on the column\textup{)},
		\item if there is $-1$ at the $i$-th row, then the entries below the $-1$ coincide with the entries on the $i$-th
		column below the diagonal $a_{ii} = -1$.
	\end{enumerate}
\end{theorem}	

\begin{example}
	Consider the following Bott matrices.
	\[
	A_1 = \begin{bmatrix}
	-1 & 0 & 0 \\ 
	1 & -1 & 0 \\
	2 & 1 & -1
	\end{bmatrix}\!,\quad
	A_2 = \begin{bmatrix}
	-1 & 0 & 0 \\
	1 & -1 & 0 \\
	1 & 0 & -1
	\end{bmatrix}\!,\quad
	A_3 = \begin{bmatrix}
	-1 & 0 & 0 \\
	-1 & -1 & 0 \\
	0 & 1	& -1
	\end{bmatrix}\!,\quad
	A_4 = \begin{bmatrix}
	-1 & 0 & 0 \\
	-1 & -1 & 0 \\
	1 & 1 & -1
	\end{bmatrix}\!.
	\]
	One can easily check that only $A_4$ satisfies all the conditions in Theorem \ref{thm_Fano} and so 
	$\B(A_4)$ is the only Fano Bott manifold among $\B(A_k)$'s for $1 \leq k \leq 4$. 
	Indeed, $A_4$ is $(10)$ on the list of Fano threefolds in the book of Oda~\cite[Figure~2.6]{Oda} and $(12)$ on the list of `Smooth toric Fano varieties'~\cite{Obro07} in the Graded Ring Database~\cite{GRDB}.
	
%	More precisely, for $A_1$, because its $(3,1)$-entry is $2 \notin \{-1,0,1\}$, $M(A_1)$ is not Fano. 
%	For $A_2$, the first column of ${A}_2+E$ has two $1$'s. Since it conflicts the second condition, $\B(A_2)$ is not Fano. For $A_3$, the first column of ${A}_3+E$ has~$-1$ 
%	at the second row. However, the $(3,1)$-entry is $0$ while the $(3,2)$-entry is~$1$. Therefore, $\B(A_3)$ is not Fano. 
\end{example}
		
%-----------------------------------------------------------------------------------------------------------------------------------------------------------------------------------------------------------------------------------------------------------------------------------------------------------
\section{Operations on Bott matrices}
\label{secOperationsOnBottMatrices}

For two Bott matrices $A$ and $A'$ in $\mathcal{M}_{n}$, we say that $A$ and $A'$ are {\em isomorphic} if $\B(A)$ and $\B(A')$ are isomorphic as toric varieties.
Equivalently, $A$ and $A'$ are isomorphic if the corresponding fans $\Sigma_A$ and $\Sigma_{A'}$ are unimodularly equivalent, i.e., there is a $\Z$-linear map in $\mathrm{GL}(n,\Z)$ which sends a maximal cone in $\Sigma_A$ to a maximal cone in $\Sigma_{A'}$. 
For a given Bott matrix $A$, there are two natural ways of producing (possibly new) isomorphic Bott matrices as follows.

For $A\in\mathcal{M}_{n}$ and $I \subset [n]$, we consider the $n \times n$ lower triangular matrix $L_I$ whose $j$-th column~${\bf c}_j$ is defined by
\[
	{\bf c}_j \coloneqq 
	\begin{cases} 
		{\bf e}_j & \text{if $j \in I$}, \\
		{\bf v}_j & \text{if $j \in I^c$},
	\end{cases}
\] where ${\bf v}_{j}$ is the $j$-th column of $A$.
\begin{proposition}\label{prop_op1}
	For $A \in \mathcal{M}_{n}$ and $I \subset [n]$, the matrix 
	\[
		A_I \coloneqq L_I^{-1} \cdot L_{I^c}
	\]
	is also a Bott matrix. Moreover, $A$ and $A_I$ are  isomorphic.
	We denote the operation $A \mapsto A_I$ by {\bf ``Op.1''}.	
\end{proposition}

\begin{proof}
	Observe that $A_I$ is a lower triangular integer matrix and 
	\[
		(A_I)_{ii} = \sum_{j=1}^n (L_I^{-1})_{ij} (L_{I^c})_{ji} = (L_I^{-1})_{ii} (L_{I^c})_{ii} = -1
	\]
	since $(L_I^{-1})_{ij} = 0$ for $j > i$ and $(L_{I^c})_{ji} = 0$ for $i > j$. Thus the first claim easily follows. 
	
	For the latter statement, consider a $\Z$-linear map given by $L_{I}^{-1} \in \mathrm{GL}(n,\Z)$. Then it induces
 a map 
	\[
		 \Sigma_A \rightarrow \Sigma_{A_I} 
	\]
	which sends each maximal cone $\sigma_J \in \Sigma_A$ to $L_I^{-1} \sigma_J \in \Sigma_{A_I}$ for each $J \subset [n]$ (in particular $\sigma_I$ to the first quadrant).
	Therefore, two fans $\Sigma_A$ and $\Sigma_{A_I}$ are unimodularly equivalent.
\end{proof}

\begin{remark}\label{rmk_op1_preserves_Fano}
The column vectors of $L_{I}$ are the ray generators of the maximal cone $\sigma_I$ in~\eqref{eq_fan}. Therefore, the operation {\oo} is nothing but a procedure of selecting a {\em reference cone} $\sigma_I$ and sending it to the {\em first quadrant} by $L_I^{-1} \in \mathrm{GL}(n,\Z)$.
Accordingly, if we take a Bott matrix $A$ which defines a \textit{Fano} Bott manifold, all the matrices $A_I$ obtained by the operation {\oo} define \textit{Fano} Bott manifolds because \textit{Fano} is an intrinsic property of a toric variety. 
\end{remark}

Depending on $A$, it could happen that a reordering (by some permutation $\pi \in S_n$) of the standard basis $\{ {\bf e}_1, \dots, {\bf e}_n \}$ changes $A$ 
into another Bott matrix (denoted by $A_\pi$) by rearranging its column vectors. Equivalently, $A_\pi = P_\pi A P_\pi^{-1}$, where $P_{\pi}$ is the row permutation matrix of $\pi$, that is, $P_{\pi}$ has~$1$ on $(i, \pi(i))$-entry for $i=1,\dots,n$ and all the other entries are zero.
We call the operation $A \mapsto A_\pi$ {\bf ``Op.2''} when $A_\pi$ is still a Bott matrix.
It is straightforward that $A$ and $A_\pi$ are isomorphic since $\Sigma_A$ and $\Sigma_{A_\pi}$ are the same up to reordering
coordinates.

\begin{example}
	Consider $A = \begin{bmatrix} -1 & 0 & 0 \\ 1 & -1 & 0 \\ 0 & 0 & -1 \end{bmatrix}$. 
	\begin{enumerate}
		\item {\em \bf (Op.1)}  For $I = \{1\} \subset [3]$, we have 
			\[
				L_I = L_I^{-1} = \begin{bmatrix} 1 & 0 & 0 \\ 0 & -1 & 0 \\ 0 & 0 & -1 \end{bmatrix}, \quad 
				L_{I^c} = \begin{bmatrix} -1 & 0 & 0 \\ 1 & 1 & 0 \\ 0 & 0 & 1 \end{bmatrix}, \quad
				A_I = \begin{bmatrix} -1 & 0 & 0 \\ -1 & -1 & 0 \\ 0 & 0 & -1 \end{bmatrix}. 
			\]

		\item {\em \bf (Op.2)} For $\pi = (2,3)$, we have 
		\[
			A_{(2,3)} = \begin{bmatrix} -1 & 0 & 0 \\ 0& -1 & 0 \\ 1 & 0 & -1 \end{bmatrix}.
		\]
	\end{enumerate}
\end{example}

The following proposition tells us that for a given Bott matrix $A$,  all Bott matrices isomorphic to $A$ are produced by applying {\oo} and {\ooo} to $A$.

\begin{proposition}\label{prop-op}
	Two Bott matrices  are isomorphic if and only if one is obtained from the other by applying two operations {\oo} and {\ooo}.
\end{proposition}

\begin{proof}
	The ``if'' part is straightforward. Thus we only need to prove the ``only if'' part. 
	
	Suppose that $A$ and $B$ are isomorphic Bott matrices, i.e., there exists a $\Z$-linear map $C \in \mathrm{GL}(n,\Z)$ sending maximal cones in $\Sigma_A$ to maximal cones in $\Sigma_B$.
	The $\Z$-linear map $C$ sends $\sigma_{[n]}^A \in \Sigma_A$ to $\sigma_I^B \in \Sigma_B$  for some $I \subset [n]$. (We denote by $\sigma_I^A$ the maximal cone in $\Sigma_A$ given by $I \subset [n]$.)  
	Since $B$ and $B_I \coloneqq (L_I^B)^{-1} \cdot L_{I^c}^B$ 
	are isomorphic, we may think of $(L_I^B)^{-1} \cdot C$ as a $\Z$-linear map which sends $\sigma_{[n]}^A \in \Sigma_A$ to $\sigma_{[n]}^{B_I} \in \Sigma_{B_I}$. 
	In other words, $(L_I^B)^{-1} \cdot C$ is a reordering of the standard basis and hence it corresponds to {\ooo}. Therefore, $B_I$ is obtained from $A$ by {\ooo}.
	Accordingly, $B$ is given by $A$ applying {\ooo} and {\oo} in order and 
	this completes the proof.
\end{proof}

Note that the rational cohomology class $y_{i}^{A}$ given in \eqref{eq_y} depends on the Bott matrix $A$. In the rest of the section, we will show that 
if $A$ and $A'$ are isomorphic, then there is a graded ring isomorphism
\begin{equation}\label{eq-iso-bott}
\varphi \colon H^{\ast} (\B(A);\Z) \to H^{\ast} (\B(A');\Z)
\end{equation} 
sending the set $\{2y_{i}^{A}\}_{i=1}^{n}$ to $\{2y_{i}^{A'}\}_{i=1}^{n}$. This fact can be obtained as a byproduct of Propositions~\ref{prop-op},~\ref{prop_iso_op1}, and~\ref{prop_iso_op2}.

\begin{proposition}\label{prop_iso_op1}
	Let $A \in  \mathcal{M}_{n}$. 
	For $I \subset [n]$, there exists a $c_1$-preserving graded ring isomophism 
	\[
	\varphi_I\colon H^{\ast}(\B(A);\Z) \to H^{\ast}(\B(A_I);\Z)
	\]
	such that $\varphi_I(2y_i^A) = 2y_i^{A_I}$ for $i = 1,\dots,n$. 
\end{proposition}

\begin{proof}
	We first consider the case where $I = \{k\}^c = [n] \setminus \{k\}$ for some $k \in [n]$ so that $A_I = L_{ \{k\}^c}^{-1} L_{ \{k\}}$.
	Let $\mathbf{a}_i$ and $\mathbf{a}^I_i$ be the $i$-th row of $A+E$ and $A_I+E$, respectively. By direct computations, we obtain
	\begin{enumerate}
		\item $\mathbf{a}^I_i=\mathbf{a}_i$ for $i<k$;
		\item $\mathbf{a}^I_i= -\mathbf{a}_k$ for $i=k$; and
		\item $\mathbf{a}^I_i= \mathbf{a}_i + a_{i,k}\mathbf{a}_k$ for $i>k$.
	\end{enumerate}
	Then the Bott tower $\B^I_{\bullet}$ corresponding to $A_I$ is given by 
	\begin{equation*}
		\B^I_i=
			\begin{cases}
				\B_i & \text{ for }i<k,\\
				P(\underline{\C}\oplus \xi_k^{-1}) & \text{ for }i=k,\\
				P(\underline{\C}\oplus (\xi_i\otimes \xi_{k}^{a_{i,k}})) & \text{ for }i>k.
			\end{cases}
	\end{equation*} 
	Here, $\xi_1,\dots,\xi_n$ are the line bundles used to construct the Bott tower $\B_{\bullet} = \B(A)$.
	
	We define a map $\varphi_I$ by 
	\begin{equation}\label{eq_op1_isom}
		\varphi_I(x_i^{A})=\begin{cases}
			x^{A_{I}}_i & \text{ for }i\neq k,\\
			x^{A_{I}}_k+\sum_{j<k} a_{k,j}x^{A_{I}}_j & \text{ for } i=k.
			\end{cases}
	\end{equation}
	Then, by $(2)$, we obtain that
	\begin{equation}\label{eq_xkA_op1}
	\varphi_I(x_k^A) = x_k^{A_I} - \alpha_k^{A_I} \quad \text{ and } \quad \varphi_I(\alpha_k^A) = - \alpha_k^{A_I}.
	\end{equation}
	
	We claim that $\varphi_I$ is well-defined and is indeed a graded ring isomorphism such that $\varphi_I(2y_i^{A})=2y^{A_{I}}_i$ for every $i=1,\dots,n$. The well-definedness follows by showing that $\varphi_I(x_i^A(x_i^A-\alpha_i^A)) = 0$ for all $i$. 
	\begin{itemize}
		\item For $i < k$, we get $\varphi_I(x_i^A(x_i^A-\alpha_i^A)) = x_i^{A_{I}} (x_i^{A_{I}} - \alpha_i^{A_{I}}) = 0$;
		\item for $i = k$, by~\eqref{eq_xkA_op1}, we have
		\[
		\varphi_I(x_k^A(x_k^A-\alpha_k^A))  = (x_k^{A_I} - \alpha_k^{A_I})(x_k^{A_I} - \alpha_k^{A_I} - \varphi_I(\alpha_k^A)) = x_k^{A_I}(x_k^{A_I} - \alpha_k^{A_I}) = 0;
		\]
	\item for $i > k$, we have $\varphi_I(x_i^A(x_i^A-\alpha_i^A)) = x_i^{A_I}(x_i^{A_I} - \varphi_I(\alpha_i^A)) = x_i^{A_{I}} (x_i^{A_{I}} - \alpha_i^{A_{I}}) = 0$;
	\end{itemize}
	where the second last equality is obtained from \eqref{eq_op1_isom} and (3):
	\begin{equation}\label{eq_alpha}
		\varphi_I(\alpha_i^{A}) =  \sum_{j<i} a_{i,j} \varphi_I(x_j^{A}) = \sum_{j\neq k, j < i} a_{i,j} x_j^{A_{I}} + a_{i,k} \left(x_k^{A_{I}} + \sum_{j < k} a_{k,j} x_j^{A_{I}} \right) = \alpha^{A_{I}}_i \qquad \text{ for } i > k.
	\end{equation}
	
	To show  $\varphi_I(2y_i^{A})=2y^{A_{I}}_i$, we only need to check the case when $i = k$ because $\varphi_I(x_i^{A}) = x_i^{A_{I}}$ and $\varphi_I(\alpha_i^{A}) = \alpha^{A_{I}}_i$ for $i \neq k$ by 
	\eqref{eq_op1_isom} and \eqref{eq_alpha}.
	Then, by~\eqref{eq_xkA_op1}, we obtain
	\[
		\varphi_I(2y_k^{A}) = \varphi_I(2x_k^{A} - \alpha_k^{A}) =
		2(x_k^{A_I}-\alpha_k^{A_I}) + \alpha_k^{A_I} 
		= 2x_k^{A_{I}} - \alpha_k^{A_{I}} = 2y_k^{A_{I}}
	\]
	and this completes the proof for the case of $I= \{k\}^c$.
	
	For a general $I = \{i_1 < \cdots < i_m\} \subset [n]$, using the fact  
	\(
		L_I = L_{\{i_1\}^c} \cdot L_{\{i_2\}^c} \cdots L_{\{i_m\}^c},
	\)
	we obtain 
	\[
	\varphi_I = \varphi_{\{i_1\}^c} \circ \varphi_{\{i_2\}^c} \circ \cdots \circ \varphi_{\{i_m\}^c}.
	\]
	Applying the previous procedure repeatedly, the result follows. One can immediately check that $\varphi_I$ is $c_1$-preserving by \eqref{eq_y}. 
\end{proof}
For the cohomology ring isomorphism between Bott manifolds induced from $\ooo$, we recall the result~\cite{CMM}.
\begin{proposition}[{\cite[Lemma~6.1]{CMM}}]\label{prop_iso_op2}
	Let $A \in \mathcal{M}_n$. 
	For a permutation $\pi$ on $[n]$, if $A_{\pi}\in \mathcal{M}_{n}$, then there is a $c_1$-preserving graded ring isomorphism
	\[
	\varphi_{\pi} \colon H^{\ast}(\B(A); \Z) \to H^{\ast}(\B(A_{\pi});\Z)
	\]
	such that $\varphi_{\pi}(x_{i}^A) = x_{\pi(i)}^{A_{\pi}}$ for $i = 1,\dots,n$. Indeed, we have $\varphi_{\pi}(2y_{i}^A) = 2y_{\pi(i)}^{A_{\pi}}$ for $i = 1,\dots,n$.
\end{proposition}

%-----------------------------------------------------------------------------------------------------------------------------------------------------------------------------------------------------------------------------------------------------------------------------------------------------------
\section{Main Theorem}
\label{secMainTheorem}

In this section, we will prove Theorem \ref{thm_main}. Throughout this section, every cohomology ring will take coefficient in $\Z$ unless otherwise stated. 

Before to begin with, we explain some notations used in this section. We use letters $\B$, $\B'$ and $\B''$ to indicate Bott manifolds. 
And also we denote by  $(x_i,y_i,\alpha_i)$, $(x_i', y_i', \alpha_i')$, and $(x_i'', y_i'', \alpha_i'')$
the elements $x$, $y$, and $\alpha$ defined in \eqref{eq_coh_bott} and \eqref{eq_y} for $\B, \B',$ and $\B''$, respectively.

\begin{lemma}\label{lem_sigma_1}
Let $\B$ and $\B'$ be Fano Bott manifolds.
If there is a $c_1$-preserving graded ring isomorphism $\varphi\colon H^\ast(\B)\to H^\ast(\B')$, then there exists a Fano Bott manifold $\B''$ together with a $c_1$-preserving 
graded ring isomorphism $\psi\colon H^*(\B) \rightarrow H^*(\B'')$ such that $\B'$ and $\B''$ are isomorphic and $\psi(x_1)=x_1''$. In particular, we have $\psi(2y_1) = 2y_1''$.
\end{lemma}
\begin{proof}
	Suppose that  Fano Bott manifolds $\B$ and $\B'$ are determined by Bott matrices  $A = (a_{ij})$ and $A' = (a_{ij}')$, respectively.
By Theorem \ref{thm_q}, there exists a permutation $\sigma$ on $[n]$ such that $\varphi(2y_i) = 2y_{\sigma(i)}'$ for each $i = 1,\dots, n$.
If $\sigma(1)\neq 1$, then we get $a'_{\sigma(1),j}=0$ for every $j<\sigma(1)$. Indeed, since $2y_1=2x_1$ and $2y_{\sigma(1)}'=2x_{\sigma(1)}'-\sum_{j<\sigma(1)} a'_{\sigma(1),j}x_j'$, we get
\[
	2(\varphi(x_1)-x_{\sigma(1)}') = -\sum_{j<\sigma(1)} a'_{\sigma(1),j}x_j'
\]
which is divisible by $2$. 
Since $a'_{\sigma(1),j}$ belongs to $\{0,\pm1\}$, we  conclude that $a'_{\sigma(1),j}=0$ for every $j<\sigma(1)$ and $\varphi(x_1)=x_{\sigma(1)}'$.
This fact tells us that the $\sigma(1)$-th row of $A'+E$ is zero, and therefore we may apply {\ooo} to $A'$ for the permutation $\pi=s_{1}s_{2}\dots s_{\sigma(1)-1}$ where $s_i$ denotes the 
simple transposition $(i,i+1)$.
Here, we note that $(\pi \circ \sigma )(1) = 1$.

We consider a Bott matrix $A'' \coloneqq A'_{\pi}$
and let $\B''$ be a  Bott manifold associated to ~$A''$. Suppose that $\B''$ is Fano. Then $\B'$ and $\B''$ are isomorphic and there is a graded ring isomorphism $\varphi'\colon H^\ast(\B')\to H^\ast(\B'')$ such that 
$\varphi'(x_{\sigma(1)}')=x_{\pi(\sigma(1))}'' = x_1''$ by Proposition~\ref{prop_iso_op2}. Then $ \varphi' \circ \varphi$ is the desired isomorphism and it completes the proof.
\end{proof}

\begin{theorem}[Theorem~\ref{thm_main}]\label{thm_main_body}
Let $\B$ and $\B'$ be Fano Bott manifolds. Assume that there is a $c_1$-preserving graded ring isomorphism $\varphi \colon H^\ast(\B)\to H^\ast(\B')$. Then $\B$ and $\B'$ are isomorphic as toric varieties.
\end{theorem}
\begin{proof}
	Suppose that  Fano Bott manifolds $\B$ and $\B'$ are determined by Bott matrices  $A = (a_{ij})$ and $A' = (a_{ij}')$, respectively.
From Theorem \ref{thm_q},
there is a permutation $\sigma$ on $[n]$ such that $\varphi(2y_i)=2y_{\sigma(i)}'$ for all $i=1,\dots,n$. We may further assume that $\varphi(x_{1})=x_{1}'$ (or equivalently $\varphi(2y_1) = 2y_1'$) 
by Lemma~\ref{lem_sigma_1}, i.e., $\sigma(1)=1$. 

Now we choose an index $k\in [n]$ (with $k > 1$) such that $\varphi(x_i)=x'_i$ for every $i<k$ and $\varphi(x_k) \neq x_k'$, that is, $\varphi(2y_{i})=2y_{i}'$ for all $i < k$. Then there are two possibilities:
\begin{enumerate}
\item $\sigma(k)=k$; or
\item $\sigma(k)>k$.
\end{enumerate}
 
We first consider the case where  $\sigma(k)=k$. That is, $\varphi(2y_{k})=2y_{k}'$ and $\varphi(x_{k})\neq x_{k}'$.
%Recall from~\eqref{eq_coh_bott} that $\alpha_i \coloneqq a_{i,1}x_1+\dots+a_{i,i-1}x_{i-1} \in H^2(\B;\Z)$. 
Since $\varphi(2y_k)=\varphi(2x_k-\alpha_k)=2\varphi(x_k)-\varphi(\alpha_k)$ and $\varphi(2y_k)=2y_k'=2x_k'-\alpha_k',$
we get 
\begin{equation}\label{eq_proof_of_main_1}
	2(\varphi(x_k)-x_k')=\varphi(\alpha_k)-\alpha_k'
\end{equation} and hence 
$\varphi(\alpha_k)-\alpha_k'$ is divisible by $2$. Note that $\varphi(\alpha_k)-\alpha_k' = \sum_{j<k}(a_{k,j}-a'_{k,j})x_j'$ by the definition of $\alpha_i$ and $\alpha_i'$ in~\eqref{eq_coh_bott}
and from the induction hypothesis. Comparing this with the equation~\eqref{eq_proof_of_main_1},  
%which implies that $\varphi(x_k)-x_k'$ has no $x_k'$ term.
%In other words, 
we have 
\(
	\varphi(x_k)=x_k'+\sum_{j<k} c_j x_j',
\)
where 
\begin{equation}\label{eq_cj}
2c_j=a_{k,j}-a'_{k,j} \qquad  \text{ for }j=1,\dots,k-1
\end{equation}
which implies that 
\begin{equation}\label{eq_cj_neq_0}
	c_j = a_{k,j}=-a_{k,j}' \quad \text{when $c_j \neq 0$}
\end{equation}
by the Fano condition in Theorem~\ref{thm_Fano}.
Note that 
\begin{equation}\label{eq_notallzero}
	\text{not every $c_j$ is zero for $j<k$ by our assumption.}
\end{equation}

Now we claim that the equation~\eqref{eq_cj_neq_0} holds even when $c_j = 0$.
By the well-definedness of the isomorphism $\varphi$, we have that
\begin{align*}
0 = \varphi(x_k(x_k-\alpha_k))&=\left(x_k'+\sum_{j<k} c_j x_j'\right)\left(x_k'+\sum_{j<k} c_j x_j' -\sum_{j<k} a_{k,j}x_j'\right)\\
&=(x_k')^2+x_k'\left(2\sum_{j<k} c_j x_j' -\sum_{j<k}a_{k,j} x_j'\right) +\left(\sum_{j<k} c_j x_j'\right)\left(\sum_{j<k} c_j x_j' -\sum_{j<k}a_{k,j} x_j'\right)\\
&=(x_k')^2+x_k'\sum_{j<k} (2c_j - a_{k,j}) x_j' +\left(\sum_{j<k} c_j x_j'\right)\left(\sum_{j<k} (c_j -a_{k,j}) x_j'\right) 
\end{align*} in $H^{\ast}(\B')$.
Since 
	\begin{itemize}	
		\item there is no term having $x_k'$ in $\left(\sum_{j<k} c_j x_j'\right)\left(\sum_{j<k} (c_j -a_{k,j}) x_j'\right)$, and 
		\item $(x_k')^2+x_k'\sum_{j<k} (2c_j - a_{k,j}) x_j'$ can be expressed as a linear combination of the linearly independent set $\{x_k'x_j'\}_{j < k}$, 
	\end{itemize}
	both $(x_k')^2+x_k'\sum_{j<k} (2c_j - a_{k,j}) x_j'$ and  $(\sum_{j<k} c_j x_j')(\sum_{j<k} (c_j -a_{k,j}) x_j')$ vanish in $H^\ast(\B')$.  
	In particular, it follows from the vanishing of the latter term above that 
\begin{align*}
0 = \left(\sum_{j<k} c_j x_j'\right)\left(\sum_{j<k} (c_j - a_{k,j}) x_j'\right)
&=\left(\sum_{j<k \atop c_j\neq 0}c_j x_j'\right)\left(-\sum_{\ell<k \atop c_\ell=0} a_{k,\ell} x_\ell'\right)=-\sum_{j,\ell<k \atop c_j\neq 0 \text{ and }c_\ell=0} c_j a_{k,\ell} x'_jx'_\ell
\end{align*} in $H^{\ast}(\B)$, where the second equality follows from \eqref{eq_cj_neq_0}.
However, $\{x'_jx'_\ell \mid c_j \neq 0, ~c_\ell = 0, ~j,\ell < k\}$ is linearly independent by Remark \ref{rem_linearly_independent} (2), so
$a_{k,\ell} = 0$ if $c_\ell = 0$. Moreover, since $2c_j = a_{k,j} - a_{k,j}'$ by \eqref{eq_cj},  we  conclude that
\begin{itemize}
	\item if $c_j = 0$, then $a_{k,j}=a_{k,j}'=0$, and 
	\item if $c_j \neq 0$, then $a_{k,j}=-a_{k,j}'$ by \eqref{eq_cj_neq_0}.
\end{itemize}
Consequently, we have $c_j = a_{k,j} = -a_{k,j}'$ for every $j < k$.
Therefore, we obtain 
\[
\varphi(x_k) = x_k' + \sum_{j<k} c_j x_j' = x_k' - \sum_{j<k} a_{k,j}' x_j' 
= x_k' - \alpha_k'.
\]

Now, we let $\B''$ be a Fano Bott manifold whose Bott matrix $A''$ is obtained from $A'$ by applying {\oo} with $I=\{k\}^c$. Then there is a $c_1$-preserving graded ring isomorphism $\varphi'\colon H^\ast(\B')\to H^\ast(\B'')$ such that $\varphi'(x_j')=x_j''$ for every $j \neq k$ and $\varphi'(x_k')=x_k''-\alpha_k''$ by \eqref{eq_op1_isom} and~\eqref{eq_xkA_op1}, and also
$\varphi'(2y_i')=2y_i''$ for all $i$ by Proposition \ref{prop_iso_op1}. 
Then,  the composition $\psi\coloneqq\varphi'\circ \varphi$ is a graded ring isomorphism $\psi\colon H^\ast(\B)\to H^\ast(\B'')$ such that $\psi(2y_i)=2y_i''$ for every $i=1,\dots,n$, $\B'$ and $\B''$ are isomorphic as toric varieties, and $\psi(x_j)=x_j''$ for every $j\leq k$. Indeed, by~\eqref{eq_xkA_op1}, we have
\[
\psi(x_k) = \varphi'(x_k' - \alpha_k') = x_k '' - \alpha_k'' + \alpha_k'' = x_k''.
\]

Now, we consider the second case, $\sigma(k)>k$. Since $\varphi(2y_k)=2\varphi(x_{k})-\sum_{j<k}a_{k,j}x_j'$ and $\varphi(2y_k)=2y_{\sigma(k)}'$, we get
\begin{equation}\label{eq_final}
	2(\varphi(x_k)-x_{\sigma(k)}')=\sum_{j<k}a_{k,j}x_j' - \sum_{\ell<\sigma(k)} a'_{\sigma(k),\ell} x_\ell'.
\end{equation}
Because $a_{k,j}$ and $a_{\sigma(k),\ell}$ belong to $\{0,1,-1\}$ and the left hand side of~\eqref{eq_final} is divisible by~$2$, we have
\[
	a'_{\sigma(k),\ell}=0   \qquad \text{ for } k\leq \ell <\sigma(k). 
\]
Indeed, the $\sigma(k)$-th row of $A'$ has consecutive zeros from $(\sigma(k), k)$ to $(\sigma(k), \sigma(k)-1)$.
Therefore by applying {\ooo} to $A'$ for the permutation $\pi=s_{k}\cdots s_{\sigma(k)-1}$, we get  a new  Bott matrix $A''$ such that $\B(A'')$ is Fano and it is isomorphic to $\B'$ and there is a graded ring isomorphism $\varphi' \colon H^*(\B') \rightarrow H^*(\B'')$ satisfying 
\begin{equation*}
\varphi'(x_{i}')= \begin{cases}
x_{i}'' &\text{ for }i<k\text{ or }i>\sigma(k),\\
x_{i+1}'' &\text{ for }k\leq i\leq \sigma(k)-1,\\
x_{k}'' &\text{ for }i=\sigma(k)
\end{cases}
\end{equation*}
by Proposition \ref{prop_iso_op2}. Since $\pi\circ\sigma(i)=i$ for every $i\leq k$, the composition $\varphi'\circ\varphi$ is a $c_1$-preserving 
graded ring isomorphism $\varphi'\circ\varphi\colon H^{\ast}(\B)\to H^{\ast}(\B'')$ satisfying 
$$\varphi'\circ\varphi(x_{i})= x_i''\quad(i<k) \quad\text{ and }\quad\varphi'\circ\varphi(2y_{k})=\varphi'(2y_{\sigma(k)}')=2y_{\pi\circ\sigma(k)}=2y_{k}''.$$
Hence this case reduces to the first case.

We may repeat the above argument as many times as necessary. Since the indices are bounded above by $n$, this process must stop. Therefore, we have a Fano Bott manifold $\widetilde{\B}$ isomorphic to $\B'$ whose Bott matrix is exactly the same as that of $\B$, i.e., $\B$ and $\B'$ are isomorphic. 
This finishes the proof.
\end{proof}

\begin{remark}\label{rem_weak_Fano}
	We cannot extend Theorem \ref{thm_main} to weak Fano Bott manifolds. Note that the Hirzerbuch surfaces $\mathcal{H}_0$ and $\mathcal{H}_2$ are weak Fano Bott manifolds.
	As we saw in Example~\ref{ex-hir}, $\mathcal{H}_0$ and $\mathcal{H}_2$ are not isomorphic but there is a $c_{1}$-preserving graded cohomology ring isomorphism between them.
\end{remark}

\begin{remark}
	One may wonder whether we can extend Theorem~\ref{thm_main} to Bott manifolds whose Bott matrices have entries $0,1,$ or $-1$. 
	However, the set of such Bott matrices is not closed under the operation~{\oo}.
	For example, consider a matrix 
	\[
	A = \begin{bmatrix}
	-1 & 0 & 0 \\ 1 & -1 & 0 \\ 1 & 1 & -1
	\end{bmatrix}.
	\]
	Then, we have that
	\[
	A_{\emptyset} = A^{-1} 
	= \begin{bmatrix}
	-1 & 0 & 0 \\ -1 & -1 & 0 \\ -2 & -1 & -1 
	\end{bmatrix}
	\]
	whose entry has $-2 \notin \{0,1,-1\}$. 
	Note that the set of Bott matrices obtained from Fano Bott manifolds is closed under the operation {\oo} as we mentioned in Remark~\ref{rmk_op1_preserves_Fano}. 	
\end{remark}

\begin{remark}
For $n = 3$, there are five Bott matrices associated to Fano Bott manifolds up to isomorphisms.

\noindent\begin{tabular}{l|ccccc}
\toprule
Oda's list & $6$ & $7$ & $8$ & $9$ & $10$\\
\O bro's list & $21$& $11$ & $18$ & $17$ & $12$ \\
\midrule 
$
\begin{array}{ll}
\text{Bott} \\
\text{matrix}
\end{array}$
& {\small$	\begin{bmatrix}
-1&0&0\\
0&-1&0\\
0&0&-1
\end{bmatrix}$}
& {\small $\begin{bmatrix}
-1&0&0\\
0&-1&0\\
1&1&-1
\end{bmatrix}$}
&
{\small $\begin{bmatrix}
-1&0&0\\
0&-1&0\\
1&-1&-1
\end{bmatrix}$}
& {\small $\begin{bmatrix}
-1&0&0\\
1&-1&0\\
0&0&-1
\end{bmatrix}$}
& {\small $\begin{bmatrix}
-1&0&0\\
1&-1&0\\
0&1&-1
\end{bmatrix}$} \\
\bottomrule
\end{tabular}

\noindent Here, we refer the lists provided by Oda~\cite[Figure~2.6]{Oda} and \O bro~\cite{Obro07} in the Graded Ring Database~\cite{GRDB}.
However, the Bott manifolds corresponding to the second and the third matrices are diffeomorphic since their cohomology rings are isomorphic as graded rings. 
Indeed, the degrees of the corresponding Bott manifolds are different, so there does not exist $c_1$-preserving cohomology ring isomorphism between them. 
Note that the family of Bott manifolds of complex dimension at most~$4$ is cohomologically rigid. See~\cite{CMS10_trans,Choi15}.
\end{remark}

%
%\bibliographystyle{amsalpha}
%\bibliography{ref}	

\begin{thebibliography}{BEM{\etalchar{+}}17}
	
	\bibitem[BEM{\etalchar{+}}17]{BEMPP}
	Victor~M. Buchstaber, Nikolay~Yu. Erokhovets, Mikiya Masuda, Taras~E. Panov,
	and Seonjeong Park, \emph{Cohomological rigidity of manifolds defined by
		$3$-dimensional polytopes}, Uspekhi Mat. Nauk \textbf{72} (2017), no.~2(434),
	3--66.
	
	\bibitem[Ber03]{Be}
	Florian Berchtold, \emph{Lifting of morphisms to quotient presentations},
	Manuscripta Math. \textbf{110} (2003), no.~1, 33--44.
	
	\bibitem[BH58]{BH58}
	Armand Borel and Friedrich Hirzebruch, \emph{Characteristic classes and
		homogeneous spaces. {I}}, Amer. J. Math. \textbf{80} (1958), 458--538.
	
	\bibitem[BK]{GRDB}
	Gavin Brown and Alexander Kasprzyk, \emph{Graded ring database},
	http://www.grdb.co.uk/forms/toricsmooth.
	
	\bibitem[BP15]{BP_toric}
	Victor~M. Buchstaber and Taras~E. Panov, \emph{Toric topology}, Mathematical
	Surveys and Monographs, vol. 204, American Mathematical Society, Providence,
	RI, 2015.
	
	\bibitem[Cho15]{Choi15}
	Suyoung Choi, \emph{Classification of {B}ott manifolds up to dimension $8$},
	Proc. Edinb. Math. Soc. (2) \textbf{58} (2015), no.~3, 653--659.
	
	\bibitem[Civ05]{Civan}
	Yusuf Civan, \emph{{B}ott towers, crosspolytopes and torus actions}, Geom.
	Dedicata \textbf{113} (2005), 55--74.
	
	\bibitem[CLMP20]{CLMP}
	Yunhyung Cho, Eunjeong Lee, Mikiya Masuda, and Seonjeong Park, {\em On Fano generalized Bott manifolds}, in preparation.


	\bibitem[CMM15]{CMM}
	Suyoung Choi, Mikiya Masuda, and Satoshi Murai, \emph{Invariance of
		{P}ontrjagin classes for {B}ott manifolds}, Algebr. Geom. Topol. \textbf{15}
	(2015), no.~2, 965--986.
	
	\bibitem[CMS10]{CMS10_trans}
	Suyoung Choi, Mikiya Masuda, and Dong~Youp Suh, \emph{Topological
		classification of generalized {B}ott towers}, Trans. Amer. Math. Soc.
	\textbf{362} (2010), no.~2, 1097--1112.
	
	\bibitem[CMS11]{CMS11}
	\bysame, \emph{Rigidity problems in toric topology: a survey}, Tr. Mat. Inst.
	Steklova \textbf{275} (2011), 188--201.
	
	\bibitem[Del88]{Del}
	Thomas Delzant, \emph{Hamiltoniens p\'{e}riodiques et images convexes de
		l'application moment}, Bull. Soc. Math. France \textbf{116} (1988), no.~3,
	315--339.
	
	\bibitem[Fan14]{Fa}
	Andrew Fanoe, \emph{Toric structures on bundles of projective spaces}, J.
	Symplectic Geom. \textbf{12} (2014), no.~4, 685--724.
	
	\bibitem[GK94]{GK_Bott_towers}
	Michael Grossberg and Yael Karshon, \emph{Bott towers, complete integrability,
		and the extended character of representations}, Duke Math. J. \textbf{76}
	(1994), no.~1, 23--58.
	
	\bibitem[Har77]{Har77}
	Robin Hartshorne, \emph{Algebraic geometry}, Springer-Verlag, New
	York-Heidelberg, 1977, Graduate Texts in Mathematics, No. 52.
	
	\bibitem[HKM20]{HKM}
	Akihiro Higashitani, Kazuki Kurimoto, and Mikiya Masuda, {\em Cohomological rigidity for toric Fano manifolds of small dimensions or large Picard numbers}, 
	in preparation.
	
	\bibitem[KKP07]{KKP}
	Yael Karshon, Liat Kessler, and Martin Pinsonnault, \emph{A compact symplectic
		four-manifold admits only finitely many inequivalent toric actions}, J.
	Symplectic Geom. \textbf{5} (2007), no.~2, 139--166.
	
	\bibitem[Kle66]{Kl}
	Steven~L. Kleiman, \emph{Toward a numerical theory of ampleness}, Ann. of Math.
	(2) \textbf{84} (1966), 293--344.
	
	\bibitem[McD11]{McD}
	Dusa McDuff, \emph{The topology of toric symplectic manifolds}, Geom. Topol.
	\textbf{15} (2011), no.~1, 145--190.
	
	\bibitem[MP08]{MP}
	Mikiya Masuda and Taras~E. Panov, \emph{Semi-free circle actions, {B}ott
		towers, and quasitoric manifolds}, Mat. Sb. \textbf{199} (2008), no.~8,
	95--122.
	
	\bibitem[MS08]{MS}
	Mikiya Masuda and Dong~Youp Suh, \emph{Classification problems of toric
		manifolds via topology}, Toric topology, Contemp. Math., vol. 460, Amer.
	Math. Soc., Providence, RI, 2008, pp.~273--286.
	
	\bibitem[Obr07]{Obro07}
	Mikkel \O bro, \emph{An algorithm for the classification of smooth {F}ano
		polytopes}, preprint, arXiv:0704.0049 (2007).
	
	\bibitem[Oda78]{Oda2}
	Tadao Oda, \emph{Torus embeddings and applications}, Tata Institute of
	Fundamental Research Lectures on Mathematics and Physics, vol.~57, Tata
	Institute of Fundamental Research, Bombay; by Springer-Verlag, Berlin-New
	York, 1978, Based on joint work with Katsuya Miyake.
	
	\bibitem[Oda88]{Oda}
	\bysame, \emph{Convex bodies and algebraic geometry}, Ergebnisse der Mathematik
	und ihrer Grenzgebiete (3) [Results in Mathematics and Related Areas (3)],
	vol.~15, Springer-Verlag, Berlin, 1988, An introduction to the theory of
	toric varieties, Translated from the Japanese.
	
	\bibitem[Pin08]{Pin}
	Martin Pinsonnault, \emph{Maximal compact tori in the {H}amiltonian group of
		4-dimensional symplectic manifolds}, J. Mod. Dyn. \textbf{2} (2008), no.~3,
	431--455.
	
	\bibitem[PT20]{PT}
	Milena Pabiniak and Susan Tolman, \emph{Symplectic cohomological rigidity via
		toric degenerations}, preprint, arXiv:2002.12434 (2020).
	
	\bibitem[Suy19]{Su}
	Yusuke Suyama, \emph{{F}ano generalized {B}ott manifolds}, Manuscripta
	Mathematica (2019).
	
\end{thebibliography}

\newcommand{\etalchar}[1]{$^{#1}$}
\providecommand{\bysame}{\leavevmode\hbox to3em{\hrulefill}\thinspace}
\providecommand{\MR}{\relax\ifhmode\unskip\space\fi MR }
% \MRhref is called by the amsart/book/proc definition of \MR.
\providecommand{\MRhref}[2]{%
	\href{http://www.ams.org/mathscinet-getitem?mr=#1}{#2}
}
\providecommand{\href}[2]{#2}

\end{document}